\documentclass[11pt]{amsart}

\usepackage[T1]{fontenc}
\usepackage[utf8]{inputenc}
\usepackage{mathpazo}
\usepackage{microtype}
\usepackage[margin=1.15in]{geometry}
\usepackage{amsmath,amssymb,mathtools}
\usepackage{xcolor}
\usepackage{tikz}
\usetikzlibrary{calc,decorations.pathreplacing}
\usepackage{booktabs,array,enumitem}
\usepackage[colorlinks,linkcolor=accent,citecolor=citec,urlcolor=accent]{hyperref}
\usepackage[capitalize,nameinlink]{cleveref}
\hypersetup{pdftitle={a-Numbers, Higher Cartier Ranks and Maximality of van der Geer--van der Vlugt Curves},
  pdfauthor={Gilberto B. Almeida Filho, Saeed Tafazolian, Stéfani C. Vieira}}

\definecolor{accent}{RGB}{24,70,120}
\definecolor{citec}{RGB}{30,110,80}
\definecolor{rc0}{RGB}{24,70,120}
\definecolor{rc1}{RGB}{46,110,168}
\definecolor{rc2}{RGB}{86,150,200}
\definecolor{rc3}{RGB}{214,120,40}
\definecolor{rc4}{RGB}{180,60,50}

\makeatletter
\def\section{\@startsection{section}{1}%
  \z@{1.5\linespacing\@plus\linespacing}{.7\linespacing}%
  {\normalfont\large\bfseries}}
\def\subsection{\@startsection{subsection}{2}%
  \z@{1.0\linespacing\@plus.7\linespacing}{.45\linespacing}%
  {\normalfont\bfseries}}
\makeatother

\newtheoremstyle{plainc}%
  {8pt}{8pt}{\itshape}{}
  {\bfseries}{.}{.5em}{}
\newtheoremstyle{defc}%
  {8pt}{8pt}{\normalfont}{}%
  {\bfseries}{.}{.5em}{}
\newtheoremstyle{remc}%
  {8pt}{8pt}{\normalfont}{}%
  {\itshape}{.}{.5em}{}

\theoremstyle{plainc}
\newtheorem{theorem}{Theorem}[section]
\newtheorem{lemma}[theorem]{Lemma}
\newtheorem{proposition}[theorem]{Proposition}
\newtheorem{corollary}[theorem]{Corollary}

\theoremstyle{defc}
\newtheorem{definition}[theorem]{Definition}
\newtheorem{example}[theorem]{Example}
\newtheorem{question}[theorem]{Question}

\theoremstyle{remc}
\newtheorem{remark}[theorem]{Remark}

\theoremstyle{plainc}
\newtheorem{mainthm}{Theorem}

\crefname{mainthm}{Theorem}{Theorems}
\crefname{remark}{Remark}{Remarks}
\crefname{question}{Question}{Questions}
\crefname{example}{Example}{Examples}



\newcommand{\Cart}{\mathcal{C}}
\newcommand{\cB}{\mathcal{B}}
\newcommand{\cU}{\mathcal{U}}
\newcommand{\F}{\mathbb{F}}
\newcommand{\Fbar}{\overline{\mathbb{F}}_p}
\newcommand{\Z}{\mathbb{Z}}
\newcommand{\Pone}{\mathbb{P}^1}
\newcommand{\Omg}{\Omega^1}
\newcommand{\bn}{\mathbf{n}}
\newcommand{\bt}{\mathbf{t}}
\newcommand{\bu}{\mathbf{u}}
\DeclareMathOperator{\rank}{rank}
\DeclareMathOperator{\Jac}{Jac}
\DeclareMathOperator{\Span}{Span}
\DeclareMathOperator{\Img}{Im}

\author{Gilberto B. Almeida Filho}
\address{Departamento de Ciências Exatas e da Terra, Universidade Federal de Mato Grosso, Várzea Grande, Mato Grosso, Brazil}
\email{gilberto.filho@ufmt.br}
\thanks{All three authors contributed equally to this work.}

\author{Saeed Tafazolian}
\address{IMECC, Universidade Estadual de Campinas, Campinas, São Paulo, Brazil}
\email{saeed@unicamp.br}

\author{Stéfani C. Vieira}
\address{Departamento de Ciências Exatas e da Terra, Universidade Federal de Mato Grosso, Várzea Grande, Mato Grosso, Brazil}
\email{stefani.vieira1@ufmt.br}

\begin{document}

\title[Cartier ranks and maximality of van der Geer--van der Vlugt curves]{$a$-Numbers, higher Cartier ranks and maximality\\
of van der Geer--van der Vlugt curves}


\begin{abstract}
Let $p$ be an odd prime and let $C_R$ be the van der Geer--van der Vlugt curve
$y^p-y=xR(x)$ over an algebraically closed field of characteristic $p$, where
$R(x)=\sum_{i=0}^{h}a_ix^{p^i}$ is an additive polynomial of degree $p^h$.
We determine the image and the rank of every power of the Cartier operator
$\Cart$ on $H^0(C_R,\Omg)$.
For $1\le r\le h-1$ the image of $\Cart^r$ is spanned by an explicit set of
monomial differentials that does not depend on $R$, and
$\rank(\Cart^r)=p^{h-r-1}\binom{p+r}{r+2}$.
The last iterate detects the linear coefficient:
$\rank(\Cart^h)=0$ if $a_0=0$ and $\rank(\Cart^h)=\binom{(p-1)/2+h}{h+1}$ if
$a_0\neq0$, so the nilpotency index of $\Cart$ is $h$ or $h+1$.
Hence, for $h\ge2$, the loci $a_0=0$ and $a_0\ne0$ have the same $a$-number, which is
the largest $a$-number of a one-point Artin--Schreier curve with pole order
$p^h+1$, but they share no Ekedahl--Oort type.
Over finite fields we deduce that a curve over $\F_{p^{2n}}$ that is geometrically
isomorphic to $C_R$ can be $\F_{p^{2n}}$-maximal or $\F_{p^{2n}}$-minimal only if
$n\ge h$, and only if $n\ge h+1$ when $a_0\neq0$.
Explicit examples show that the second bound is attained for
$(p,h)=(3,1)$, $(3,2)$ and $(5,2)$.
\end{abstract}

\maketitle

\section{Introduction}\label{sec-intro}

Let $C$ be a smooth projective curve of genus $g$ over an algebraically closed field
$k$ of characteristic $p>0$, and let $\Cart$ be the Cartier operator on
$H^0(C,\Omg)$.
The $a$-number $a(\Jac(C))=\dim_k\ker\Cart=g-\rank(\Cart)$ and the $p$-rank, which
is the stable rank of the powers of $\Cart$, are the most basic invariants of the
group scheme $\Jac(C)[p]$ \cite{Pries2008}.
More generally, the whole sequence $(\rank\Cart^r)_{r\ge1}$ is an invariant of the
Ekedahl--Oort type of $\Jac(C)[p]$ \cite{Oort2001,Moonen2001}; it carries more
information than the $a$-number, but in general it does not determine the
Ekedahl--Oort type.

Throughout the paper $p$ is odd.
For an integer $h\ge0$ let $C_R$ be the smooth projective curve with affine model
\begin{equation}\label{eq-family}
  C_R:\qquad y^p-y=xR(x),\qquad
  R(x)=\sum_{i=0}^{h}a_i x^{p^i},\qquad a_h\ne0,
\end{equation}
of genus $g(C_R)=p^h(p-1)/2$.
In characteristic $2$ these curves were introduced by van der Geer and van der Vlugt
in connection with Reed--Muller codes and the existence of supersingular curves
\cite{vdGvdV1992,vdGvdV1995}.
For odd $p$, Bouw, Ho, Malmskog, Scheidler, Srinivasan and Vincent
\cite{BHMSSV2016} studied their automorphisms, quotient curves, Jacobian
decompositions and zeta functions, and proved that $C_R$ is supersingular.
Takeuchi and Tsushima expressed the Frobenius eigenvalues of these curves through
Gauss sums \cite{TakeuchiTsushima2024}; building on this, Ito, Tatematsu and
Tsushima gave criteria for maximality and minimality over finite fields
\cite{ItoTatematsuTsushima2026}, partly recovering results of \"Ozbudak and
Sayg\i{} \cite{OzbudakSaygi2016}.
In characteristic $2$, Ito, Takeuchi and Tsushima determined the $L$-polynomials
\cite{ItoTakeuchiTsushima2026}.
Supersingularity fixes the Newton polygon of $C_R$, but it does not determine the
ranks of the powers of $\Cart$.

Explicit Cartier computations are available for several related curves.
Pries gave the $a$-number of the monomial curves $y^p-y=x^d$
\cite[Remark~3.4]{Pries2005}, Pries and Weir determined the Ekedahl--Oort types of
Hermitian curves \cite{PriesWeir2015}, Farnell and Pries found families of
Artin--Schreier curves with Cartier--Manin matrix of constant rank under the
condition $p\equiv1$ modulo the pole orders \cite{FarnellPries2013}, and Montanucci
and Speziali treated Fermat and Hurwitz curves \cite{MontanucciSpeziali2018}.
Garcia and Tafazolian used the Cartier operator to obtain restrictions on maximal
curves \cite{GarciaTafazolian2008}.
For general Artin--Schreier covers, Booher and Cais proved upper and lower bounds for
the $a$-number in terms of the ramification \cite{BooherCais2020}; the lower bound
is attained for $p=3,5$ by \cite{AbneyMcPeek2022}, and Cais and Ulmer showed that it
is the $a$-number of a general curve $y^p-y=f(x)$ with $\deg f$ prime to $p$
\cite{CaisUlmer2026}.
Groen extended the Booher--Cais bounds to the kernels of all powers of $\Cart$
\cite[Theorem~6.16]{Groen2024}. For one-point covers of $\Pone$, the proof of his
upper bound identifies the monomial equation $y^p-y=x^d$ as the extremal case.

\subsection*{Main results}
The aim of this paper is to determine the ranks of all powers of $\Cart$ uniformly
for every additive polynomial $R$ in \eqref{eq-family}, and to draw the arithmetic
consequences over finite fields.
Put $M=(p-1)/2$.
For $h\ge2$ and $1\le r\le h-1$ we define in \eqref{eq-Uhr} a subspace
$\cU_h^{(r)}\subseteq H^0(C_R,\Omg)$ spanned by monomial differentials
$x^iy^s\,dx$; it depends on $p$, $h$ and $r$, but not on the coefficients of $R$.

\begin{mainthm}\label{thm-A}
Let $h\ge2$ and $1\le r\le h-1$. Then $\Img(\Cart^r)=\cU_h^{(r)}$ and
\[
  \rank(\Cart^r)=p^{\,h-r-1}\binom{p+r}{r+2}.
\]
In particular
$\rank(\Cart)=p^{h-1}(p^2-1)/6$ and
$a(\Jac(C_R))=p^{h-1}(p-1)(2p-1)/6$, independently of $a_0,\dots,a_{h-1}$.
\end{mainthm}

\begin{mainthm}\label{thm-B}
Let $h\ge1$. Then $\Cart^{h+1}=0$ on $H^0(C_R,\Omg)$ and
\[
  \rank(\Cart^h)=
  \begin{cases}
    0, & a_0=0,\\[2pt]
    \dbinom{M+h}{h+1}, & a_0\neq0.
  \end{cases}
\]
Consequently the nilpotency index of $\Cart$ is exactly $h$ if $a_0=0$ and exactly
$h+1$ if $a_0\neq0$.
For $h=1$ this gives $a(\Jac(C_R))=g$ if $a_0=0$ and
$a(\Jac(C_R))=(p-1)(3p-1)/8$ if $a_0\ne0$.
\end{mainthm}

For $h=0$ one has $\Cart=0$ (\cref{prop-h0}).
\Cref{thm-A,thm-B} therefore give the complete sequence of Cartier ranks
(\cref{cor-sequence,tab-sequences}); it depends only on $p$, $h$ and on whether
$a_0=0$.
Three consequences are worth recording.
First, for $h\ge1$ the loci $a_0=0$ and $a_0\ne0$ share no Ekedahl--Oort type, so
the vanishing of $a_0$ is preserved by every isomorphism between members of the
family (\cref{rem-EO}).
Second, for $h\ge2$ both loci have the same $a$-number, and by Groen's bounds this is
the largest $a$-number among all curves $y^p-y=f(x)$ with $\deg f=p^h+1$
(\cref{cor-extremal-ranks}); the rank sequence, however, does not characterize the
family (\cref{rem-special}).
Third, the nilpotency index gives an obstruction to maximality over finite fields.
Recall that a curve $X$ of genus $g$ over $\F_\ell$, with $\ell$ a square, is
\emph{$\F_\ell$-maximal} (resp.\ \emph{$\F_\ell$-minimal}) if
$\#X(\F_\ell)=\ell+1+2g\sqrt\ell$ (resp.\ $\ell+1-2g\sqrt\ell$).

\begin{mainthm}\label{thm-C}
Let $h\ge1$ and $n\ge1$, and let $X$ be a curve over $\F_{p^{2n}}$ such that
$X\otimes\Fbar\cong C_R$ for some $R\in\Fbar[x]$ as in \eqref{eq-family}.
If $X$ is $\F_{p^{2n}}$-maximal or $\F_{p^{2n}}$-minimal, then $n\ge h$; if
moreover $a_0\ne0$, then $n\ge h+1$.
\end{mainthm}

The bound $n\ge h$ is attained by the curves with $a_0=\dots=a_{h-1}=0$ and
$a_h^{p^h}+a_h=0$, which are $\F_{p^{2h}}$-maximal
\cite{OzbudakSaygi2016,ItoTatematsuTsushima2026}.
The bound $n\ge h+1$ is attained for $(p,h)=(3,1)$, $(3,2)$ and $(5,2)$
(\cref{ex-sharp}).
For maximal curves given by an equation \eqref{eq-family} over a finite field, the
inequality $n\ge h$ was proved by Ito, Tatematsu and Tsushima using the explicit
$L$-polynomial \cite[Proposition~4.13]{ItoTatematsuTsushima2026}.
\Cref{thm-C} extends it to minimal curves and to all forms of $C_R$, and refines it
according to $a_0$; its proof uses only \cref{thm-B} and the Garcia--Tafazolian
criterion \cite{GarciaTafazolian2008}.

\subsection*{Relation to Groen's bounds}
Our proofs are self-contained and elementary: they use an explicit monomial basis, a
closed formula for $\Cart$ (\cref{prop-expansion}), and block-triangular minors.
The lower bound $\rank(\Cart^r)\ge p^{h-r-1}\binom{p+r}{r+2}$ for $r\le h-1$ can
also be obtained from Groen's inequality combined with the monomial case
(\cref{rem-groen}).
The new ingredients are the explicit, coefficient-free description of
$\Img(\Cart^r)$, which yields the matching upper bound; the analysis of the last
iterate, where the coefficient $a_0$ appears; and the consequences over finite
fields.

\subsection*{Organization}
\Cref{sec-prelim} collects background and notation.
\Cref{sec-basis,sec-expansion} give a basis of regular differentials and the
expansion of $\Cart$.
\Cref{sec-small} treats $h=0$ and the Hermitian case.
\Cref{sec-first} computes $\Img(\Cart)$ for $h\ge2$, \cref{sec-higher} proves
\cref{thm-A}, and \cref{sec-last} proves \cref{thm-B}.
\Cref{sec-consequences} discusses Ekedahl--Oort types, extremality and
characteristic~$2$, and \cref{sec-finite} proves \cref{thm-C}.

\section{Preliminaries}\label{sec-prelim}

Throughout, $k$ is an algebraically closed field of characteristic $p>2$, and all
curves are smooth, projective and irreducible.

\subsection{Artin--Schreier curves}

\begin{proposition}[{\cite[Proposition~3.7.8]{Stichtenoth2009}}]\label{prop-AS}
Let $f\in k[x]$ be a polynomial of degree $d$ with $p\nmid d$. Then the curve
$C\colon y^p-y=f(x)$ has a unique point $P_\infty$ above $x=\infty$; the cover
$C\to\Pone_x$ is totally ramified at $P_\infty$ and unramified elsewhere, and
\[
  v_{P_\infty}(x)=-p,\qquad v_{P_\infty}(y)=-d,\qquad
  g(C)=\frac{(p-1)(d-1)}{2}.
\]
Moreover $\operatorname{div}(dx)=(2g-2)P_\infty$
\cite[Proposition~6.4.1]{Stichtenoth2009}.
\end{proposition}

For the family \eqref{eq-family} we fix the notation
\[
  f(x)=xR(x)=\sum_{\nu=0}^{h}a_\nu x^{p^\nu+1},\qquad d=p^h+1,\qquad
  M=\frac{p-1}{2},\qquad \alpha_\nu=a_\nu^{1/p}\in k .
\]
Since $p\nmid d$, \cref{prop-AS} gives $g(C_R)=p^h(p-1)/2$ and
$v_{P_\infty}(y)=-(p^h+1)$.

\begin{remark}[Normalizations]\label{rem-normalization}
Over $k$, the substitution $x=\xi X$ with $\xi^{p^h+1}=a_h^{-1}$ normalizes the
leading coefficient to $1$ and preserves the vanishing of $a_0$; we nevertheless
keep $a_h$ in the calculations.
More precisely, replacing $xR(x)$ by $xR(x)+A(x)+c$, with $A$ additive and $c\in k$,
does not produce a new isomorphism class outside the family over $k$.
Indeed, modulo Artin--Schreier coboundaries $w^p-w$, every term $bx^{p^i}$ reduces to
$b^{1/p^i}x$, and constants can be removed.
To remove a linear term $cx$, substitute $x=X+\lambda$: modulo coboundaries and
constants,
\[
  xR(x)\equiv XR(X)+\Phi(\lambda)X,
  \qquad \Phi(\lambda)=R(\lambda)+\sum_{i=0}^{h}(a_i\lambda)^{1/p^i}.
\]
For $h\ge1$ the function $\lambda\mapsto\Phi(\lambda)^{p^h}$ is a polynomial of degree
$p^{2h}$ with leading coefficient $a_h^{p^h}$, and for $h=0$ one has
$\Phi(\lambda)=2a_0\lambda$.
In both cases $\Phi$ is surjective on $k$, so $\Phi(\lambda)=-c$ has a solution.
Over a finite field these normalizations require extracting roots and are not
available in general; this is why \cref{thm-C} is stated for curves that are only
geometrically isomorphic to some $C_R$.
\end{remark}

\subsection{The Cartier operator}

Let $F=k(C)$ and let $x$ be a separating element. Every $\phi\in F$ has a unique
expansion $\phi=\sum_{s=0}^{p-1}\phi_s^p\,x^s$ with $\phi_s\in F$.

\begin{definition}[{\cite{Cartier1957,Serre1958}}]\label{def-cartier}
The \emph{Cartier operator} $\Cart\colon\Omega_F\to\Omega_F$ is
$\Cart\big(\sum_{s=0}^{p-1}\phi_s^p\,x^s\,dx\big)=\phi_{p-1}\,dx$.
It does not depend on the choice of $x$, and it maps $H^0(C,\Omg)$ to itself.
\end{definition}

We use only the following consequences of the definition:
\begin{align}
  \Cart(\omega+\eta)&=\Cart(\omega)+\Cart(\eta), &
  \Cart(\phi^p\omega)&=\phi\,\Cart(\omega), \label{eq-C1}\\
  \Cart(x^m\,dx)&=
  \begin{cases}
    x^{(m+1)/p-1}\,dx, & m\equiv p-1\pmod p,\\
    0, & m\not\equiv p-1\pmod p.
  \end{cases} \label{eq-C2}
\end{align}

\begin{remark}[Semilinearity]\label{rem-semilinear}
By \eqref{eq-C1}, $\Cart$ is $p^{-1}$-linear and $\Cart^r$ is $p^{-r}$-linear.
Hence the image of a $k$-subspace under $\Cart$ is again a $k$-subspace, and
$\Img(\Cart^{r+1})=\Cart(\Img\Cart^r)$.
If $\omega_1,\dots,\omega_g$ is a basis of $H^0(C,\Omg)$ and
$\Cart(\omega_j)=\sum_i m_{ij}\omega_i$, then the matrix of $\Cart^r$ in this basis
is $\mathsf M\,\mathsf M^{(1/p)}\cdots\mathsf M^{(1/p^{r-1})}$, where
$\mathsf M=(m_{ij})$ and $\mathsf M^{(1/p^e)}=(m_{ij}^{1/p^e})$.
By $\rank(\Cart^r)$ we mean $\dim_k\Img(\Cart^r)$, which is the rank of this matrix.
For a curve over a perfect field these ranks are computed after extension to an
algebraic closure, and they are invariants of the curve over that closure.
\end{remark}

\begin{proposition}[see {\cite{Pries2008}}]\label{prop-a-number}
For a curve $C$ of genus $g$ one has $a(\Jac(C))=g-\rank(\Cart)$, and the $p$-rank
of $\Jac(C)$ equals $\rank(\Cart^N)$ for $N\gg0$.
\end{proposition}

The isomorphism class of the group scheme $A[p]$ of a principally polarized abelian
variety $A$ is classified by its \emph{Ekedahl--Oort type}
\cite{Oort2001,Moonen2001}. It determines the ranks of all iterates of Frobenius and
Verschiebung on the Dieudonn\'e module of $A[p]$; for $A=\Jac(C)$ the sequence
$(\rank\Cart^r)_{r\ge1}$ is therefore an invariant of the Ekedahl--Oort type.
One has $0\le a(A)\le g$, and $a(A)=g$ if and only if $A$ is \emph{superspecial},
i.e.\ isomorphic to a product of supersingular elliptic curves \cite{Pries2008}.

\subsection{Combinatorial tools}

\begin{lemma}\label{lem-binom-nonzero}
If $0\le t\le s\le p-1$, then $\binom{s}{t}\not\equiv0\pmod p$.
\end{lemma}
\begin{proof}
The numerator of $\binom{s}{t}=\frac{s!}{t!(s-t)!}$ has no factor $p$.
\end{proof}

\begin{lemma}\label{lem-trinomial}
For integers $j\ge\tau+q\ge0$ with $\tau,q\ge0$,
\[
  \binom{j}{\tau}\binom{j-\tau}{q}=\frac{j!}{\tau!\,q!\,(j-\tau-q)!}
  =\binom{j}{\tau+q}\binom{\tau+q}{\tau}.
\]
\end{lemma}

\begin{lemma}[Binomial Vandermonde matrices]\label{lem-vandermonde}
Let $\lambda_1,\dots,\lambda_n\in\{0,1,\dots,p-1\}$ be distinct and let
$0\le D\le p-1$. Then the $(D+1)\times n$ matrix
$\big(\binom{\lambda_\beta}{t}\big)_{0\le t\le D,\ 1\le\beta\le n}$ has rank
$\min\{D+1,n\}$ over $\F_p$.
\end{lemma}
\begin{proof}
Since $t\le D<p$, $\binom{X}{t}=X^t/t!+(\text{terms of lower degree})$ with $t!$
invertible modulo $p$. Hence the matrix equals $\Lambda\cdot(\lambda_\beta^t)$ for an
invertible lower triangular matrix $\Lambda$, and $(\lambda_\beta^t)$ has rank
$\min\{D+1,n\}$ because its maximal square minors formed by the first rows are
Vandermonde determinants of distinct elements of $\F_p$.
\end{proof}

\begin{lemma}[{\cite{Stanley2011}}]\label{lem-stars}
For integers $n\ge1$ and $K\ge0$, the number of $(w_1,\dots,w_n)\in\Z_{\ge0}^n$ with
$w_1+\dots+w_n\le K$ is $\binom{K+n}{n}$.
\end{lemma}

\begin{lemma}[Block-triangular matrices]\label{lem-blocktri}
Let $\Theta$ be a matrix over a field whose rows and columns are partitioned into
groups indexed by the same finite totally ordered set $(I,\preceq)$; some groups may
be empty. Write $\Theta_{J'J}$ for the block with rows in group $J'$ and columns in
group $J$, and assume that $\Theta_{J'J}=0$ whenever $J\prec J'$
\textup{(}or whenever $J'\prec J$\textup{)}.
\begin{enumerate}[label=\textup{(\roman*)}, leftmargin=2em]
  \item If every diagonal block $\Theta_{JJ}$ has full row rank, then $\Theta$ has
    full row rank.
  \item If every $\Theta_{JJ}$ has full column rank, then $\Theta$ has full column
    rank.
  \item If every $\Theta_{JJ}$ is square and invertible, then $\Theta$ is invertible.
\end{enumerate}
Here a block with no rows has full row rank, and a block with no columns has full
column rank.
\end{lemma}

\begin{proof}
For (i), choose in each $\Theta_{JJ}$ a set of columns on which it restricts to an
invertible square matrix $S_J$ (no columns if $\Theta_{JJ}$ has no rows). The
submatrix of $\Theta$ formed by all rows and the chosen columns is square and block
triangular with diagonal blocks $S_J$, so its determinant is
$\prod_J\det S_J\neq0$. Part (ii) is (i) applied to the transpose, and (iii) follows
from (i) and (ii).
\end{proof}

\section{A basis of regular differentials}\label{sec-basis}

\begin{lemma}\label{lem-basis}
Assume $h\ge1$. A basis of $H^0(C_R,\Omg)$ is
\[
  \cB_h=\Big\{\omega_{i,j}=x^iy^j\,dx\ :\ 0\le j\le p-2,\ \ 0\le i\le p^{h-1}(p-1-j)-1\Big\}.
\]
\end{lemma}

\begin{proof}
By \cref{prop-AS}, $dx$ has neither zeros nor poles on the affine part of $C_R$, where
$x$ and $y$ are regular. Hence $x^iy^j\,dx$ with $i,j\ge0$ is regular away from
$P_\infty$, and it is regular at $P_\infty$ if and only if
\[
  -pi-(p^h+1)j+p^h(p-1)-2\ge0 .
\]
If $j\ge p-1$, then $pi+(p^h+1)j\ge p^h(p-1)+p-1$, so this fails. For
$0\le j\le p-2$ the inequality reads $pi\le p\cdot p^{h-1}(p-1-j)-(j+2)$ with
$1\le j+2\le p$, which is equivalent to $i\le p^{h-1}(p-1-j)-1$. The number of
differentials obtained is
\[
  \sum_{j=0}^{p-2}p^{h-1}(p-1-j)=p^{h-1}\,\frac{p(p-1)}{2}=g .
\]
They are linearly independent because $1,y,\dots,y^{p-1}$ are linearly independent
over $k(x)$.
\end{proof}

\section{The Cartier expansion}\label{sec-expansion}

Fix $\omega_{i,j}=x^iy^j\,dx\in\cB_h$. Since $y=y^p-f(x)$,
\begin{equation}\label{eq-yj}
  y^j=\sum_{\tau=0}^{j}(-1)^{j-\tau}\binom{j}{\tau}(y^\tau)^p f(x)^{j-\tau}.
\end{equation}
Put $n=j-\tau$. For a multi-index $\bn=(n_0,\dots,n_h)\in\Z_{\ge0}^{h+1}$ with
$|\bn|=n_0+\dots+n_h=n$, the corresponding monomial of $f(x)^n$ has coefficient
$\binom{n}{n_0,\dots,n_h}\prod_{\nu}a_\nu^{n_\nu}$, and the corresponding monomial of
$x^if(x)^n$ has exponent
\begin{equation}\label{eq-m}
  \mu=i+\sum_{\nu=0}^{h}n_\nu(p^\nu+1)=i+n+\sum_{\nu=0}^{h}n_\nu p^\nu .
\end{equation}

\begin{remark}
The factor $x^i$ must be merged with the monomial of $f(x)^n$ \emph{before}
\eqref{eq-C2} is applied: by \eqref{eq-C1}, only $p$-th powers can be pulled through
$\Cart$. Thus the survival condition concerns the total exponent $\mu$ modulo $p$.
\end{remark}

Write $i=pL+c$ with $0\le c\le p-1$. Then $\mu\equiv c+n+n_0\pmod p$, and since
$c\le p-1$ and $n_0\le n\le j\le p-2$,
\[
  0\le c+n+n_0\le 3p-5 .
\]
Hence, by \eqref{eq-C2}, a term survives $\Cart$ only if
\begin{equation}\label{eq-survival}
  c+n+n_0=ep-1\qquad\text{for some }e\in\{1,2\}.
\end{equation}
In that case $\mu=p\big(L+e+\sum_{\nu\ge1}n_\nu p^{\nu-1}\big)-1$, and we obtain the
formula on which everything else rests.

\begin{proposition}[Cartier expansion]\label{prop-expansion}
With the notation above,
\[
  \Cart(\omega_{i,j})=\sum_{\tau=0}^{j}\ \sum_{\bn}
  (-1)^{j-\tau}\binom{j}{\tau}\binom{n}{n_0,\dots,n_h}
  \prod_{\nu=0}^{h}\alpha_\nu^{n_\nu}\ x^{u}y^{\tau}\,dx,
\]
where $\bn$ runs over the multi-indices with $|\bn|=n=j-\tau$ satisfying
\eqref{eq-survival}, and
\begin{equation}\label{eq-u}
  u=L+e-1+\sum_{\nu=1}^{h}n_\nu p^{\nu-1}.
\end{equation}
\end{proposition}

\begin{proof}
Multiply \eqref{eq-yj} by $x^i\,dx$, expand $f(x)^n$ by the multinomial theorem, pull
$(y^\tau)^p$ out using \eqref{eq-C1}, and apply \eqref{eq-C2} to each monomial
$x^\mu\,dx$. Integer coefficients lie in $\F_p$ and are fixed by $p$-th roots.
\end{proof}

\begin{remark}
The set of monomials that can occur in $\Cart(\omega_{i,j})$ does not depend on $R$;
the coefficients of $R$ enter only through the weights $\prod\alpha_\nu^{n_\nu}$,
some of which vanish when a coefficient of $R$ is zero.
\end{remark}

\section{Small heights}\label{sec-small}

\subsection{The case \texorpdfstring{$h=0$}{h=0}}
Here $C_R\colon y^p-y=a_0x^2$ with $a_0\neq0$, and $g=M=(p-1)/2$. Now
$v_{P_\infty}(y)=-2$ and $v_{P_\infty}(dx)=p-3$, so $x^iy^j\,dx$ is regular if and
only if $pi+2j\le p-3$, and a basis of regular differentials is
$dx,\ y\,dx,\ \dots,\ y^{M-1}\,dx$.
For $0\le j\le M-1$, every power of $x$ occurring in the expansion \eqref{eq-yj} of
$y^j$ has exponent $2(j-\tau)\le 2j\le p-3$, so it is never $\equiv p-1\pmod p$.

\begin{proposition}\label{prop-h0}
For $h=0$, $\Cart=0$ on $H^0(C_R,\Omg)$; thus $\Jac(C_R)$ is superspecial and
$a(\Jac(C_R))=g=(p-1)/2$.
\end{proposition}

\begin{remark}
The curve $y^p-y=a_0x^2$ is hyperelliptic. Ekedahl proved that a superspecial
hyperelliptic curve has genus at most $(p-1)/2$ \cite{Ekedahl1987}, so these curves
attain his bound.
\end{remark}

\subsection{The case \texorpdfstring{$h=1$}{h=1} and \texorpdfstring{$a_0=0$}{a0=0}}
For $h=1$, \cref{lem-basis} gives $\cB_1=\{x^iy^j\,dx:\ i,j\ge0,\ i+j\le p-2\}$.

\begin{proposition}\label{prop-h1a0}
If $h=1$ and $a_0=0$, then $\Cart=0$ and $a(\Jac(C_R))=g=p(p-1)/2$.
\end{proposition}

\begin{proof}
Here $f(x)=a_1x^{p+1}$, so in the term of \eqref{eq-yj} with $n=j-\tau$ the exponent
of $x$ is $i+n(p+1)\equiv i+n\pmod p$. Survival would require
$i+n\equiv p-1\pmod p$, which is impossible because $0\le i+n\le i+j\le p-2$.
\end{proof}

\begin{remark}\label{rem-hermitian}
Over $k$, the curve $y^p-y=a_1x^{p+1}$ is isomorphic to the Hermitian curve
$Y^p+Y=X^{p+1}$ (put $y=\zeta Y$ with $\zeta^{p-1}=-1$ and rescale $x$). Its genus
$p(p-1)/2$ is the largest genus of a superspecial curve \cite{Ekedahl1987}, and its
Ekedahl--Oort type is determined in \cite{PriesWeir2015}.
The case $h=1$, $a_0\ne0$ is the case $h=1$ of \cref{thm-B}, proved in
\cref{sec-last}.
\end{remark}

\section{The first iterate for \texorpdfstring{$h\ge2$}{h at least 2}}\label{sec-first}

In this section $h\ge2$. We determine $\Img(\Cart)$; this is also the base case of
the induction in \cref{sec-higher}.

\subsection{A universal image space}
Define
\begin{equation}\label{eq-Uh}
  \cU_h=\Span_k\left\{x^{p^{h-1}t+z}y^\tau\,dx\ :\
  \begin{array}{l}\tau,t\ge0,\ \ \tau+t\le p-2,\\[2pt]
  0\le z\le p^{h-2}(p-1-\tau-t)-1\end{array}\right\}.
\end{equation}

\begin{lemma}\label{lem-Uh-basis}
The differentials displayed in \eqref{eq-Uh} are pairwise distinct elements of
$\cB_h$; in particular they form a basis of $\cU_h$.
\end{lemma}

\begin{proof}
For $x^{p^{h-1}t+z}y^\tau\,dx$ we have $i=p^{h-1}t+z$ and $j=\tau\le p-2$. Since
$z\le p^{h-2}(p-1-\tau-t)-1\le p^{h-1}(p-1-\tau-t)-1$, we get
$i\le p^{h-1}(p-1-\tau)-1$, which is the bound of \cref{lem-basis}. Moreover
$z\le p^{h-2}(p-1)-1<p^{h-1}$, so $t$ and $z$ are the quotient and the remainder of
$i$ upon division by $p^{h-1}$; together with $\tau=j$ this shows that the triple
$(\tau,t,z)$ is determined by the differential.
\end{proof}

\begin{lemma}\label{lem-image}
$\Cart\big(H^0(C_R,\Omg)\big)\subseteq\cU_h$.
\end{lemma}

\begin{proof}
Consider a surviving term of $\Cart(\omega_{i,j})$ given by $\tau$ and $\bn$, with
$|\bn|=n=j-\tau$, $i=pL+c$ and $c+n+n_0=ep-1$. By \eqref{eq-u} its exponent is
$u=p^{h-1}t+z$ with
\[
  t=n_h,\qquad z=L+e-1+\sum_{\nu=1}^{h-1}n_\nu p^{\nu-1}.
\]
By \cref{lem-basis}, $i\le p^{h-1}(p-1-j)-1$, hence $L\le p^{h-2}(p-1-j)-1$. Also
$\sum_{\nu=1}^{h-1}n_\nu p^{\nu-1}\le p^{h-2}\sum_{\nu=1}^{h-1}n_\nu=p^{h-2}(n-t-n_0)$.
Since $j=\tau+n$,
\[
  z\le p^{h-2}(p-1-\tau-t-n_0)-1+(e-1).
\]
If $e=1$ this is at most $p^{h-2}(p-1-\tau-t)-1$. If $e=2$, then $c+n+n_0=2p-1$ and
$c+n\le 2p-3$ force $n_0\ge2$, so $e-1=1\le p^{h-2}n_0$ and the same bound follows.
Finally $\tau+t\le\tau+n=j\le p-2$.
\end{proof}

\begin{proposition}\label{prop-dimU}
$\dim\cU_h=p^{h-1}(p^2-1)/6$. Consequently $\rank(\Cart)\le p^{h-1}(p^2-1)/6$.
\end{proposition}

\begin{proof}
Grouping the pairs $(\tau,t)$ by $s=\tau+t$,
\[
  \dim\cU_h=p^{h-2}\sum_{\tau+t\le p-2}(p-1-\tau-t)
  =p^{h-2}\sum_{s=0}^{p-2}(s+1)(p-1-s)=p^{h-2}\cdot\frac{p(p^2-1)}{6}. \qedhere
\]
\end{proof}

\subsection{A triangular minor}
For each triple $(\tau,t,z)$ occurring in \eqref{eq-Uh} define
\begin{equation}\label{eq-eta}
  \eta_{\tau,t,z}=x^{pz+p-1-t}\,y^{\tau+t}\,dx .
\end{equation}

\begin{lemma}\label{lem-eta-regular}
The differentials $\eta_{\tau,t,z}$ are pairwise distinct elements of $\cB_h$.
\end{lemma}

\begin{proof}
From $z\le p^{h-2}(p-1-\tau-t)-1$ we get
$pz+p-1-t\le p^{h-1}(p-1-(\tau+t))-1$, which is the bound of \cref{lem-basis} for
$j=\tau+t$. Since $0\le p-1-t\le p-1$, the remainder of $i=pz+p-1-t$ modulo $p$
determines $t$, the quotient determines $z$, and then $\tau=j-t$.
\end{proof}

By \cref{lem-Uh-basis,lem-eta-regular}, both families are indexed by the same set of
triples $(\tau,t,z)$. Consider the square submatrix of the Cartier--Manin matrix whose
row $(\tau',t',z')$ corresponds to the basis element $x^{p^{h-1}t'+z'}y^{\tau'}dx$ of
$\cU_h$ and whose column $(\tau,t,z)$ corresponds to $\eta_{\tau,t,z}$. By
\cref{lem-image}, these rows contain the whole image of each selected column.

\begin{lemma}\label{lem-triangular}
\begin{enumerate}[label=\textup{(\roman*)}, leftmargin=2em]
  \item The entry in row $(\tau',t',z')$ and column $(\tau,t,z)$ vanishes unless
    $\tau'+t'\le\tau+t$.
  \item If $\tau'+t'=\tau+t$, the entry vanishes unless $(\tau',t',z')=(\tau,t,z)$,
    and the diagonal entry equals $(-1)^t\binom{\tau+t}{\tau}\alpha_h^{\,t}\neq0$.
\end{enumerate}
Consequently the selected square submatrix is invertible.
\end{lemma}

\begin{proof}
Fix the column $\eta_{\tau,t,z}$, whose $y$-exponent is $j=\tau+t$. A surviving term
has output $y$-exponent $\tau'$, and we put $n'=j-\tau'$. By the proof of
\cref{lem-image} its output exponent is $p^{h-1}n_h+z'$ with $0\le z'<p^{h-1}$, so it
lands in the row with $t'=n_h\le n'$. Hence $\tau'+t'\le\tau'+n'=j=\tau+t$, which is
(i). Suppose equality holds. Then $n_h=n'$ and $n_0=\dots=n_{h-1}=0$. For
$i=pz+p-1-t$ we have $c=p-1-t$, and the survival condition becomes
$p-1-t+n'\equiv p-1\pmod p$, that is, $\tau\equiv\tau'\pmod p$.
Since $0\le\tau,\tau'\le p-2$, this forces $\tau'=\tau$, hence $n'=t'=t$, $e=1$, and
the output exponent is
\[
  \frac{(pz+p-1-t)+t(p^h+1)-(p-1)}{p}=z+p^{h-1}t .
\]
So the only term of this column with $\tau'+t'=\tau+t$ lands in the row $(\tau,t,z)$,
with coefficient $(-1)^t\binom{\tau+t}{\tau}\alpha_h^{\,t}$, which is nonzero by
\cref{lem-binom-nonzero} (as $\tau+t\le p-2$) and $a_h\neq0$. This proves (ii).

For the last assertion, group rows and columns by $J=\tau+t\in\{0,\dots,p-2\}$. By
(i) the block with row group $J'$ and column group $J$ vanishes when $J'>J$, and by
(ii) each diagonal block is a diagonal matrix with nonzero diagonal entries. Now
\cref{lem-blocktri}(iii) applies.
\end{proof}

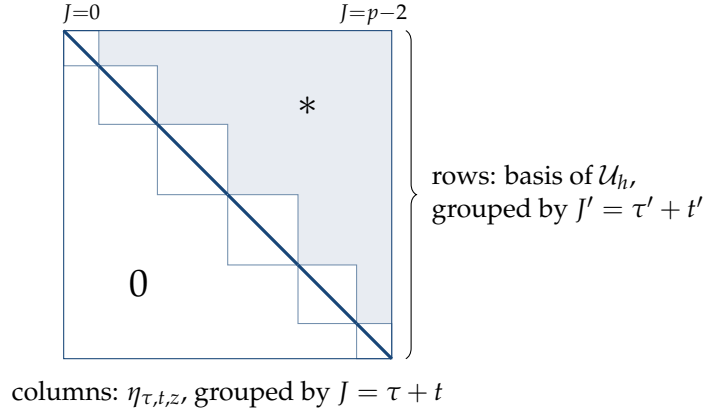
\begin{figure}[ht]
\centering
\begin{tikzpicture}[scale=0.62]
  \def\b{{0,0.75,2.0,3.5,5.0,6.25,7.0}}
  \fill[accent!10] (0,0) rectangle (7,7);
  \foreach \J in {0,...,5}{
    \pgfmathsetmacro{\xa}{\b[\J]}
    \pgfmathsetmacro{\xb}{\b[\J+1]}
    \fill[white] (\xa,0) rectangle (\xb,7-\xb);
    \fill[white] (\xa,7-\xa) rectangle (\xb,7-\xb);
    \draw[accent!60] (\xa,7-\xa) rectangle (\xb,7-\xb);
    \draw[accent,very thick] (\xa,7-\xa) -- (\xb,7-\xb);
  }
  \draw[accent] (0,0) rectangle (7,7);
  \node at (5.2,5.4) {\Large $\ast$};
  \node at (1.6,1.6) {\Large $0$};
  \node[font=\scriptsize] at (0.375,7.35) {$J{=}0$};
  \node[font=\scriptsize] at (6.6,7.35) {$J{=}p{-}2$};
  \draw[decorate,decoration={brace,amplitude=4pt}] (7.3,7) -- (7.3,0)
     node[midway,right=6pt,align=left,font=\small]
     {rows: basis of $\cU_h$,\\ grouped by $J'=\tau'+t'$};
  \node[below,font=\small,align=center] at (3.5,-0.25)
     {columns: $\eta_{\tau,t,z}$, grouped by $J=\tau+t$};
\end{tikzpicture}
\caption{Schematic shape of the minor in \cref{lem-triangular} (drawn for $p=7$; block
$J$ has size $(J+1)p^{h-2}(p-1-J)$). Entries occur only in row groups $J'\le J$, and
each diagonal block is itself diagonal, with entries
$(-1)^t\binom{\tau+t}{\tau}\alpha_h^{t}\neq0$.}
\label{fig-triangular}
\end{figure}

\begin{theorem}\label{thm-hge2}
For every $h\ge2$, $\Img(\Cart)=\cU_h$. Hence
\[
  \rank(\Cart)=\frac{p^{h-1}(p^2-1)}{6}\qquad\text{and}\qquad
  a(\Jac(C_R))=\frac{p^{h-1}(p-1)(2p-1)}{6}.
\]
\end{theorem}

\begin{proof}
\Cref{lem-image,prop-dimU} give $\Img(\Cart)\subseteq\cU_h$, and \cref{lem-triangular}
provides an invertible square submatrix of size $\dim\cU_h$, so
$\rank(\Cart)\ge\dim\cU_h$. Hence $\Img(\Cart)=\cU_h$, and
\[
  a(\Jac(C_R))=g-\rank(\Cart)=\frac{p^h(p-1)}{2}-\frac{p^{h-1}(p^2-1)}{6}
  =\frac{p^{h-1}(p-1)(2p-1)}{6}. \qedhere
\]
\end{proof}

\begin{remark}\label{rem-lower-coeff}
Only $a_h\neq0$ is used in the proof. The lower coefficients create additional
entries of the Cartier--Manin matrix, but in the minor of \cref{lem-triangular} they
lie strictly above the diagonal blocks, so they cannot lower the rank, and
\cref{lem-image} prevents them from raising it. As \cref{thm-B} shows, $a_0$ becomes
relevant exactly at the last iterate $\Cart^h$; for $h=1$ that iterate is $\Cart$
itself.
\end{remark}

\begin{remark}
For $h\ge2$ one has $a(\Jac(C_R))/g(C_R)=(2p-1)/(3p)$. This ratio does not depend on
$h$, equals $5/9$ for $p=3$, and increases to $2/3$ as $p\to\infty$. Thus the
Jacobian is never superspecial, but its $a$-number always exceeds $g/2$.
\end{remark}

\begin{remark}[The monomial member]\label{rem-monomial}
If $a_0=\dots=a_{h-1}=0$, only $\bn=(0,\dots,0,n)$ occurs in
\cref{prop-expansion}, and \eqref{eq-survival} reads $c+n=p-1$ because
$c+n\le 2p-3$. Writing $n_c=p-1-c$, we obtain
\[
  \Cart(\omega_{i,j})=
  \begin{cases}
    (-1)^{n_c}\binom{j}{n_c}\alpha_h^{\,n_c}\ \omega_{L+n_cp^{h-1},\,j-n_c}, & n_c\le j,\\[3pt]
    0, & n_c>j.
  \end{cases}
\]
For $h\ge2$ one has $L<p^{h-1}$, so distinct basis elements with nonzero image have
distinct images, and the Cartier--Manin matrix is a weighted partial permutation
matrix; the same then holds for every $\Cart^r$.
When the lower-order terms $a_\nu x^{p^\nu+1}$, $\nu<h$, are switched on, many new
entries appear; \cref{thm-hge2} and \cref{thm-A} say that they do not change the
ranks of $\Cart^r$ for $r\le h-1$.
\end{remark}

\section{Higher iterates: proof of \texorpdfstring{\cref{thm-A}}{Theorem A}}\label{sec-higher}

In this section $h\ge2$. For $1\le r\le h-1$ put $m=h-r-1\ge0$ and define
\begin{equation}\label{eq-Uhr}
  \cU_h^{(r)}=\Span_k\left\{
  x^{\,v+p^mb+\sum_{\ell=1}^{r}p^{m+\ell}t_\ell}\,y^s\,dx\ :\
  \begin{array}{l}0\le v<p^m,\quad b,t_1,\dots,t_r,s\ge0,\\[2pt]
  b+t_1+\dots+t_r+s\le p-2\end{array}\right\}.
\end{equation}
For $r=1$ (so $m=h-2$) this is the space $\cU_h$ of \eqref{eq-Uh}: the correspondence
is $t=t_1$, $\tau=s$ and $z=v+p^{h-2}b$, which identifies the pairs $(v,b)$ with
$0\le v<p^{h-2}$, $0\le b\le p-2-\tau-t$ with the integers
$0\le z\le p^{h-2}(p-1-\tau-t)-1$.

\begin{lemma}\label{lem-Uhr-basis}
Let $1\le r\le h-1$ and $m=h-r-1$.
\begin{enumerate}[label=\textup{(\roman*)}, leftmargin=2em]
  \item Every differential displayed in \eqref{eq-Uhr} belongs to $\cB_h$.
  \item Distinct parameter tuples $(v,b,t_1,\dots,t_r,s)$ give distinct elements of
    $\cB_h$.
\end{enumerate}
Consequently these differentials form a basis of $\cU_h^{(r)}$, and
\begin{equation}\label{eq-dimUhr}
  \dim\cU_h^{(r)}=p^{\,h-r-1}\binom{p+r}{r+2}.
\end{equation}
\end{lemma}

\begin{proof}
(i) The differential has $i=v+p^mb+\sum_{\ell=1}^{r}p^{m+\ell}t_\ell$ and
$j=s\le p-2$. Since $v\le p^m-1$ and every weight $p^m,\dots,p^{m+r}$ is at most
$p^{m+r}=p^{h-1}$,
\[
  i\le p^m-1+p^{h-1}(b+t_1+\dots+t_r)\le p^m-1+p^{h-1}(p-2-s)
  \le p^{h-1}(p-1-s)-1,
\]
where the last step uses $p^m\le p^{h-1}$. This is the bound of \cref{lem-basis}.

(ii) The integer $v$ is the remainder of $i$ modulo $p^m$, and
$(i-v)/p^m=b+t_1p+t_2p^2+\dots+t_rp^r$ with all coefficients at most $p-2$. This is
the base-$p$ expansion of $(i-v)/p^m$, so $b,t_1,\dots,t_r$ are determined by $i$,
and $s=j$.

For the dimension, $v$ takes $p^m=p^{h-r-1}$ values, and by \cref{lem-stars} the
$r+2$ variables $b,t_1,\dots,t_r,s$ with sum at most $p-2$ admit $\binom{p+r}{r+2}$
choices.
\end{proof}

\begin{theorem}\label{thm-higher}
Let $h\ge2$ and $1\le r\le h-1$. Then $\Img(\Cart^r)=\cU_h^{(r)}$, and therefore
$\rank(\Cart^r)=p^{\,h-r-1}\binom{p+r}{r+2}$.
\end{theorem}

This is \cref{thm-A}; the formula for $r=1$ and the $a$-number are
\cref{thm-hge2}.

\begin{proof}
For $r=1$ the assertion is \cref{thm-hge2}.
Assume $1\le r\le h-2$ and $\Img(\Cart^r)=\cU_h^{(r)}$, and set $m=h-r-1\ge1$.
By \cref{rem-semilinear}, $\Img(\Cart^{r+1})=\Cart(\cU_h^{(r)})$, so it suffices to
show that $\Cart(\cU_h^{(r)})=\cU_h^{(r+1)}$.

\emph{Containment.} Take a basis differential
\begin{equation}\label{eq-omega-r}
  \omega=x^{\,v+p^mb+\sum_{\ell=1}^{r}p^{m+\ell}t_\ell}\,y^s\,dx
\end{equation}
of $\cU_h^{(r)}$ and write $v=pv'+c$ with $0\le c\le p-1$ and $0\le v'<p^{m-1}$.
Choose a term of the Cartier expansion of $\omega$, given by $0\le\sigma\le s$,
$n=s-\sigma$ and $|\bn|=n$. Since $m\ge1$, the residue modulo $p$ of the exponent of
$x$ is $c$, and as in \cref{sec-expansion} a surviving term satisfies
\begin{equation}\label{eq-survival-r}
  c+n+n_0=ep-1,\qquad e\in\{1,2\}.
\end{equation}
After applying $\Cart$, the exponent of $x$ is
\begin{equation}\label{eq-exp-r}
  v'+p^{m-1}b+\sum_{\ell=1}^{r}p^{m+\ell-1}t_\ell+(e-1)+\sum_{\nu=1}^{h}n_\nu p^{\nu-1}.
\end{equation}
Put $\Delta=e-1+\sum_{\nu=1}^{m-1}n_\nu p^{\nu-1}$ and write
$v'+\Delta=v''+\kappa p^{m-1}$ with $0\le v''<p^{m-1}$. We claim that
$\kappa\in\{0,1\}$. For $m=1$ this is immediate, because $v'=0$ and $\Delta=e-1$.
For $m\ge2$ we have $\Delta\le1+(p-2)p^{m-2}<p^{m-1}$, and the claim follows from
$v'<p^{m-1}$. Using $h=m+r+1$, \eqref{eq-exp-r} becomes
\[
  v''+p^{m-1}(b+n_m+\kappa)+\sum_{\ell=1}^{r}p^{m-1+\ell}(t_\ell+n_{m+\ell})
  +p^{m+r}n_h ,
\]
which has the shape of \eqref{eq-Uhr} for $\cU_h^{(r+1)}$, with output $y$-exponent
$\sigma$. It remains to check the defining inequality. Since
$\sum_{\nu=1}^hn_\nu=n-n_0$ and $\sigma=s-n$,
\[
  (b+n_m+\kappa)+\sum_{\ell=1}^{r}(t_\ell+n_{m+\ell})+n_h+\sigma
  =b+\sum_{\ell=1}^{r}t_\ell+s-n_0-\sum_{\nu=1}^{m-1}n_\nu+\kappa .
\]
Here $b+\sum_\ell t_\ell+s\le p-2$, and
$\kappa\le n_0+\sum_{\nu=1}^{m-1}n_\nu$: this is clear if $\kappa=0$; if $\kappa=1$
and $e=1$, then some $n_\nu$ with $1\le\nu\le m-1$ is nonzero, since otherwise
$\Delta=0$ and there is no carry; if $e=2$, then \eqref{eq-survival-r} and
$c+n\le 2p-3$ give $n_0\ge2$. Thus $\Cart(\cU_h^{(r)})\subseteq\cU_h^{(r+1)}$.
In the displayed exponent the coefficients $b+n_m+\kappa$ and $t_\ell+n_{m+\ell}$ are
at most $p-2$ by the inequality just proved, so no further carries occur and the top
coefficient of the output is exactly $n_h$. We record this:
\begin{equation}\label{eq-top-digit}
  \text{the output lies in the row of }\cU_h^{(r+1)}\text{ with }t_{r+1}=n_h
  \text{ and }y\text{-exponent }\sigma.
\end{equation}

\emph{Lower bound.} Take a target basis element
\[
  x^{\,v'+p^{m-1}b+\sum_{\ell=1}^{r+1}p^{m-1+\ell}t_\ell}\,y^\sigma\,dx\in\cU_h^{(r+1)}
\]
and associate with it the source differential
\begin{equation}\label{eq-source}
  x^{\,pv'+(p-1-t_{r+1})+p^mb+\sum_{\ell=1}^{r}p^{m+\ell}t_\ell}\,
  y^{\sigma+t_{r+1}}\,dx .
\end{equation}
It is a basis element of $\cU_h^{(r)}$: its parameters are
$v=pv'+(p-1-t_{r+1})<p^m$, $b$, $t_1,\dots,t_r$ and $s=\sigma+t_{r+1}$, and
$b+t_1+\dots+t_{r+1}+\sigma\le p-2$. Different targets give different sources,
because $v$ determines $t_{r+1}$ (through its residue $p-1-t_{r+1}$ modulo $p$) and
$v'$, and then $s$ determines $\sigma$; by \cref{lem-Uhr-basis}(ii) the sources are
distinct basis elements. Consider the square submatrix of $\Cart|_{\cU_h^{(r)}}$ with
rows indexed by the targets and columns by the associated sources, and group both by
$J=\sigma+t_{r+1}$ (for a column, $J$ is its $y$-exponent $s$).

By the containment, every surviving term of a selected source lands in one of the
rows. For a surviving term of a source with $y$-exponent $J$ and output $y$-exponent
$\sigma'$ we have $n_h\le n=J-\sigma'$, so by \eqref{eq-top-digit} it lands in a row
with $J'=\sigma'+n_h\le J$. If $J'=J$, then $n_h=n$ and $n_0=\dots=n_{h-1}=0$, and
the survival condition becomes
$(p-1-t_{r+1})+(J-\sigma')=p-1+\sigma-\sigma'$, which can equal $p-1$ but not
$2p-1$. Hence $\sigma'=\sigma$, $n_h=t_{r+1}$, $e=1$, $\Delta=0$ and $\kappa=0$, so
the output is exactly the associated target, with coefficient
$(-1)^{t_{r+1}}\binom{\sigma+t_{r+1}}{\sigma}\alpha_h^{\,t_{r+1}}\neq0$.
Thus the block with row group $J'$ and column group $J$ vanishes when $J'>J$, and
each diagonal block is diagonal with nonzero diagonal entries. By
\cref{lem-blocktri}(iii) the submatrix is invertible, so
$\rank\big(\Cart|_{\cU_h^{(r)}}\big)\ge\dim\cU_h^{(r+1)}$. Together with the
containment, $\Cart(\cU_h^{(r)})=\cU_h^{(r+1)}$, and induction completes the proof.
\end{proof}

\begin{remark}[An alternative route to the lower bound]\label{rem-groen}
For a polynomial $f\in k[x]$ of degree $d=p^h+1$ write $\Cart_f$ for the Cartier
operator of $y^p-y=f(x)$. Groen's upper bound for $\dim_k\ker\Cart_f^r$
\cite[Theorem~6.16]{Groen2024} depends only on the ramification data. For a one-point
cover of $\Pone$ with break $d$, the proof of that bound is sharp for the monomial
member $y^p-y=x^d$. Hence
$\rank(\Cart_f^r)\ge\rank(\Cart_{x^d}^r)$ for all $r\ge1$. At the monomial point the
ranks can be read off from \cref{rem-monomial}, and this gives the inequality
$\rank(\Cart^r)\ge p^{h-r-1}\binom{p+r}{r+2}$ for every member $C_R$. The containment
$\Img(\Cart^r)\subseteq\cU_h^{(r)}$, which gives the reverse inequality and the
independence of the image from $R$, does not follow from \cite{Groen2024}.
\end{remark}

\section{The last iterate: proof of \texorpdfstring{\cref{thm-B}}{Theorem B}}\label{sec-last}

In this section $h\ge1$. For $h=1$ put $\cU_1^{(0)}=H^0(C_R,\Omg)$. With this
convention, \cref{thm-higher} gives $\Img(\Cart^{h-1})=\cU_h^{(h-1)}$ for every
$h\ge1$. At $r=h-1$ (so $m=0$ and $v=0$) this space has the basis
\begin{equation}\label{eq-last-basis}
  x^{\,b+pt_1+p^2t_2+\dots+p^{h-1}t_{h-1}}\,y^s\,dx,\qquad
  b+t_1+\dots+t_{h-1}+s\le p-2 ;
\end{equation}
for $h=1$ this is the basis $\cB_1$. For a differential $\omega$ in
\eqref{eq-last-basis} we write $\bt=(t_1,\dots,t_{h-1})$,
$T=|\bt|=t_1+\dots+t_{h-1}$ (so $T=0$ if $h=1$) and
\[
  \rho(\omega)=b+T+2s-(p-1).
\]

\begin{lemma}\label{lem-last-setup}
Let $\omega$ be as in \eqref{eq-last-basis}, and consider a term of its Cartier
expansion given by $0\le\sigma\le s$, $n=s-\sigma$ and $|\bn|=n$.
\begin{enumerate}[label=\textup{(\roman*)}, leftmargin=2em]
  \item The term survives only if $b+n+n_0=p-1$; in that case $n_0\ge1$.
  \item A surviving term equals
  \[
    (-1)^{n}\binom{s}{\sigma}\binom{n}{n_0,\dots,n_h}\prod_{\nu=0}^{h}\alpha_\nu^{n_\nu}\;
    x^{u_1+pu_2+\dots+p^{h-1}u_h}\,y^\sigma\,dx,
  \]
  where $u_\nu=t_\nu+n_\nu$ for $1\le\nu\le h-1$ and $u_h=n_h$; moreover
  $u_1+\dots+u_h+2\sigma=\rho(\omega)$.
  \item In particular $\Cart(\omega)=0$ if $\rho(\omega)<0$, and always
  $\rho(\omega)\le p-3$. The integers $u_1,\dots,u_h$ are at most $p-3$, so the
  output monomials are pairwise distinct and determine $(u_1,\dots,u_h,\sigma)$.
  \item Given $\omega$ and a vector $(u_1,\dots,u_h,\sigma)$, at most one term of the
  expansion of $\Cart(\omega)$ lands on that vector.
\end{enumerate}
\end{lemma}

\begin{proof}
Write $\omega=x^iy^s\,dx$ with $i=b+pt_1+\dots+p^{h-1}t_{h-1}$. By \eqref{eq-m}, the
term indexed by $\sigma$ and $\bn$ is a constant multiple of
$(y^\sigma)^p\,x^\mu\,dx$ with
\[
  \mu=b+pt_1+\dots+p^{h-1}t_{h-1}+n+\sum_{\nu=0}^{h}n_\nu p^\nu .
\]

(i) Modulo $p$ only $b$, $n$ and $n_0$ survive, so $\mu\equiv b+n+n_0\pmod p$. Since
$n_0\le n\le s$, $b+s\le b+T+s\le p-2$ and $s\le p-2$,
\[
  0\le b+n+n_0\le b+2s\le 2p-4 .
\]
The only integer in $[0,2p-4]$ congruent to $p-1$ modulo $p$ is $p-1$, so survival
forces $b+n+n_0=p-1$. Since $b+n\le b+s\le p-2$, we get $n_0\ge1$.

(ii) The coefficient is that of \cref{prop-expansion}. By (i),
$\mu=p\big(\sum_{\nu=1}^{h-1}p^{\nu-1}(t_\nu+n_\nu)+p^{h-1}n_h\big)+p-1$, which gives
the output exponent. Moreover $n_1+\dots+n_h=n-n_0$ and, by (i), $n_0=p-1-b-n$, so
\[
  \sum_{\nu=1}^{h-1}(t_\nu+n_\nu)+n_h=T+(n-n_0)=T+2n+b-(p-1).
\]
Since $2n=2s-2\sigma$, this gives $u_1+\dots+u_h+2\sigma=\rho(\omega)$.

(iii) The left-hand side of the last identity is nonnegative, which gives the first
claim. Also $\rho(\omega)\le(p-2-s)+2s-(p-1)=s-1\le p-3$. Hence each $u_\nu\le p-3$,
and $u_1+pu_2+\dots+p^{h-1}u_h$ is a base-$p$ expansion.

(iv) The vector fixes $\sigma$, hence $n=s-\sigma$, and then $n_\nu=u_\nu-t_\nu$ for
$1\le\nu\le h-1$, $n_h=u_h$ and $n_0=n-\sum_{\nu\ge1}n_\nu$.
\end{proof}

Fix $0\le\rho\le p-3$. Let $V_\rho$ be the span of the basis elements
\eqref{eq-last-basis} with $\rho(\omega)=\rho$, and $W_\rho$ the span of the
differentials $x^{u_1+pu_2+\dots+p^{h-1}u_h}y^\sigma\,dx$ with
$u_1+\dots+u_h+2\sigma=\rho$. By \cref{lem-last-setup},
$\Cart(V_\rho)\subseteq W_\rho$. Put
\begin{equation}\label{eq-AR}
  N_\rho=\#\{(u_1,\dots,u_h,\sigma)\in\Z_{\ge0}^{h+1}:\ u_1+\dots+u_h+2\sigma=\rho\}
  =\sum_{\sigma=0}^{\lfloor \rho/2\rfloor}\binom{\rho-2\sigma+h-1}{h-1},
\end{equation}
so that $\dim W_\rho=N_\rho$. These monomials are regular: their $x$-exponent is at
most $p^{h-1}(\rho-2\sigma)\le p^{h-1}(p-1-\sigma)-1$ and $\sigma\le p-2$, so they
satisfy the bound of \cref{lem-basis}. The sequence $N_\rho$ is nondecreasing in
$\rho$, because $(u_1,\dots,u_h,\sigma)\mapsto(u_1+1,u_2,\dots,u_h,\sigma)$ is
injective.

\begin{lemma}\label{lem-VR}
A basis element \eqref{eq-last-basis} lies in $V_\rho$ if and only if $s\ge\rho+1$
and $b=p-1+\rho-T-2s\ge0$. Consequently $\dim V_\rho=N_{p-3-\rho}$.
\end{lemma}

\begin{proof}
If $b+T+2s=p-1+\rho$, then $b+T+s\le p-2$ is equivalent to $s\ge\rho+1$. Writing
$s=\rho+1+\sigma'$, the remaining condition becomes $b+T+2\sigma'=p-3-\rho$ with
$b,t_1,\dots,t_{h-1},\sigma'\ge0$: these are $h$ variables of weight $1$ and one of
weight $2$, so their number is $N_{p-3-\rho}$.
\end{proof}

\begin{figure}[ht]
\centering
\begin{tikzpicture}[scale=0.82]
  \draw[->,gray!70] (-0.4,0) -- (6.8,0) node[right,black] {$b=i$};
  \draw[->,gray!70] (0,-0.4) -- (0,6.6) node[above,black] {$s=j$};
  \foreach \x in {0,...,5} \node[below,font=\scriptsize] at (\x,-0.12) {$\x$};
  \foreach \y in {1,...,5} \node[left,font=\scriptsize] at (-0.12,\y) {$\y$};
  \draw[dashed,gray] (0,5.35) -- (5.35,0);
  \node[font=\scriptsize,gray,rotate=-45,anchor=south] at (2.85,2.85) {$i+j=p-2$};
  \draw[rc0,thick] (4,1) -- (0,3);
  \draw[rc1,thick] (3,2) -- (1,3);
  \draw[rc2,thick] (2,3) -- (0,4);
  \foreach \i in {0,...,5}{
    \foreach \j in {0,...,5}{
      \pgfmathtruncatemacro{\s}{\i+\j}
      \ifnum\s<6
        \pgfmathtruncatemacro{\r}{\i+2*\j-6}
        \ifnum\r<0
          \fill[gray!40] (\i,\j) circle (2.3pt);
        \else
          \fill[rc\r] (\i,\j) circle (3.4pt);
        \fi
      \fi
    }}
  \begin{scope}[shift={(7.6,5.4)},font=\small]
    \foreach \r/\d/\w [count=\n from 0] in {0/3/1,1/2/1,2/2/2,3/1/2,4/1/3}{
      \fill[rc\r] (0,-0.62*\n) circle (3.4pt);
      \node[right] at (0.2,-0.62*\n) {$\rho=\r$:\ $\dim V_{\rho}=\d,\ \dim W_{\rho}=\w$};
    }
    \fill[gray!40] (0,-3.3) circle (2.3pt);
    \node[right,align=left] at (0.2,-3.3) {$\rho<0$: killed by $\Cart$};
  \end{scope}
\end{tikzpicture}
\caption{The grading of \cref{lem-last-setup} for $h=1$ and $p=7$. The basis
element $x^iy^j\,dx$ of $\cB_1$ has $\rho=i+2j-(p-1)$; the points on the line
$i+2j=p-1+\rho$ span $V_\rho$, which $\Cart$ maps into $W_\rho$. Summing
$\min\{\dim V_\rho,\dim W_\rho\}=1+1+2+1+1=6=\tfrac{p^2-1}{8}$ gives $\rank(\Cart)$
when $a_0\ne0$.}
\label{fig-h1}
\end{figure}
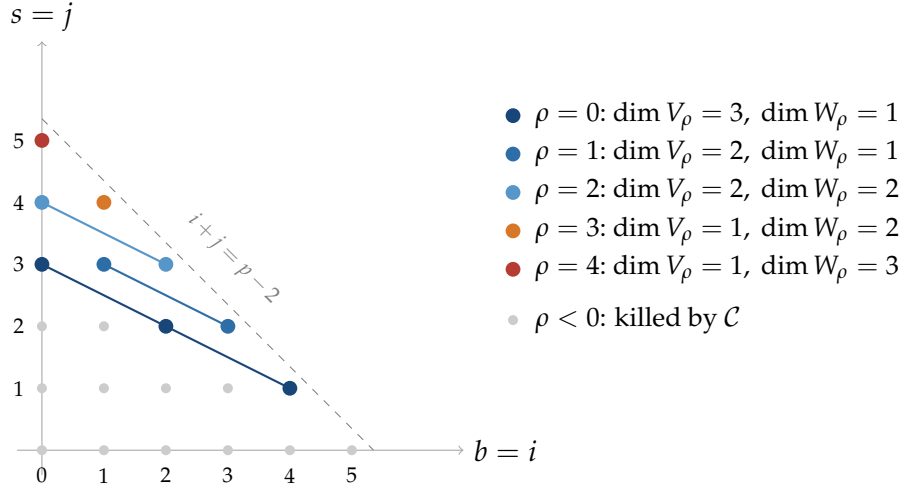

\begin{proposition}\label{prop-block-rank}
Assume $a_0\neq0$ and $0\le\rho\le p-3$. Then
\[
  \rank\big(\Cart|_{V_\rho}\colon V_\rho\to W_\rho\big)=\min\{N_\rho,\,N_{p-3-\rho}\}.
\]
\end{proposition}

\begin{proof}
Let $\Theta$ be the matrix of $\Cart|_{V_\rho}$. Its rows are indexed by vectors
$(\bu',u_h,\sigma)$ with $\bu'=(u_1,\dots,u_{h-1})$, and its columns by the sources
$(b,\bt,s)$. By \cref{lem-last-setup}(ii),(iv), the entry in row $(\bu',u_h,\sigma)$
and column $(b,\bt,s)$ is either $0$ or the single term with $n_\nu=u_\nu-t_\nu$
($1\le\nu\le h-1$). In particular it vanishes unless $t_\nu\le u_\nu$ for all
$1\le\nu\le h-1$.

Group the rows according to $\bu'$ and the columns according to $\bt$; both
groupings are indexed by the finite set $I=\{\bt\in\Z_{\ge0}^{h-1}:|\bt|\le p-3\}$,
allowing empty groups. Fix a total order $\preceq$ on $I$ refining the componentwise
partial order (for instance, order by $|\bt|$ and then lexicographically). If the
block with row group $\bu'$ and column group $\bt$ is nonzero, then $\bt\le\bu'$
componentwise, hence $\bt\preceq\bu'$. So this block vanishes whenever
$\bu'\prec\bt$, which is the hypothesis of \cref{lem-blocktri}. The diagonal block
$\Theta_{\bt\bt}$ consists of the rows with $\bu'=\bt$ and the columns with the given
$\bt$.

Fix $\bt$ and put $T=|\bt|$. The rows of $\Theta_{\bt\bt}$ are the vectors
$(\bt,u_h,\sigma)$ with $u_h+2\sigma=\rho-T$, so their number is
\[
  \beta_T=
  \begin{cases}\lfloor(\rho-T)/2\rfloor+1,& T\le\rho,\\ 0,& T>\rho.\end{cases}
\]
By \cref{lem-VR} the columns are indexed by $s$ with
$\rho+1\le s\le\lfloor(p-1+\rho-T)/2\rfloor$; since $p-1=2M$, their number is
\[
  \gamma_T=\max\Big\{0,\ M-\rho+\Big\lfloor\frac{\rho-T}{2}\Big\rfloor\Big\}.
\]
In particular, if $T\le\rho$ and $\gamma_T>0$, then $\gamma_T-\beta_T=M-\rho-1$.
On $\Theta_{\bt\bt}$ we have $n_1=\dots=n_{h-1}=0$, $n_h=\rho-T-2\sigma$ and, by
\cref{lem-last-setup}(i), $n_0=n-n_h=s+\sigma-\rho+T\ge1$. The multinomial
coefficient has only two nonzero parts, so by \cref{lem-trinomial} the
$(\sigma,s)$-entry is
\begin{align*}
  &(-1)^{s-\sigma}\binom{s}{\sigma}\binom{s-\sigma}{\rho-T-2\sigma}
  \alpha_0^{\,s+\sigma-\rho+T}\alpha_h^{\,\rho-T-2\sigma}\\
  &\qquad=\underbrace{(-1)^{\sigma}\tbinom{\rho-T-\sigma}{\sigma}
     \alpha_0^{\,\sigma}\alpha_h^{\,\rho-T-2\sigma}}_{\text{depends only on the row}}
  \cdot\binom{s}{\rho-T-\sigma}\cdot
  \underbrace{(-1)^{s}\alpha_0^{\,s-(\rho-T)}}_{\text{only on the column}}.
\end{align*}
The row and column factors are nonzero because $a_0a_h\neq0$,
$s-(\rho-T)\ge1$, and $\rho-T-\sigma\le p-3$
(\cref{lem-binom-nonzero}). Hence $\rank\Theta_{\bt\bt}$ equals
the rank of $\big(\binom{s}{\theta}\big)_{\theta,s}$, where $\theta=\rho-T-\sigma$ runs
over $\theta_0,\dots,\rho-T$ with $\theta_0=\lceil(\rho-T)/2\rceil$. Writing
$\theta=\theta_0+t'$ and using \cref{lem-trinomial} in the form
$\binom{s}{\theta_0+t'}=\binom{s}{\theta_0}\binom{s-\theta_0}{t'}\big/\binom{\theta_0+t'}{\theta_0}$,
where the removed factors are nonzero because $\theta_0<\rho+1\le s\le p-2$ and
$\theta_0+t'\le\rho\le p-3$, we are reduced to the matrix
$\big(\binom{s-\theta_0}{t'}\big)$ with $0\le t'\le\lfloor(\rho-T)/2\rfloor$. The
values $s-\theta_0$ are distinct elements of $\{0,\dots,p-1\}$, so
\cref{lem-vandermonde} gives $\rank\Theta_{\bt\bt}=\min\{\beta_T,\gamma_T\}$.

Suppose first that $\rho\le M-1$. If $T>\rho$ then $\beta_T=0$; if $T\le\rho$ then
$\gamma_T=\beta_T+(M-\rho-1)\ge\beta_T$. In both cases $\Theta_{\bt\bt}$ has full row
rank, so by \cref{lem-blocktri}(i) the matrix $\Theta$ has full row rank $N_\rho$.

Suppose now that $\rho\ge M$. Every column $(b,\bt,s)$ satisfies $s\ge\rho+1$ and
$b\ge0$, hence $T=p-1+\rho-b-2s\le p-3-\rho=2M-2-\rho\le M-2<\rho$. So every $\bt$
with $\gamma_T>0$ has $T<\rho$, and then $\gamma_T=\beta_T+(M-\rho-1)<\beta_T$; thus
$\Theta_{\bt\bt}$ has full column rank. Blocks with $\gamma_T=0$ have full column rank
trivially. By \cref{lem-blocktri}(ii) the matrix $\Theta$ has full column rank
$\dim V_\rho=N_{p-3-\rho}$.

Finally, $\rho\le M-1$ is equivalent to $\rho\le p-3-\rho$, so in both cases the rank
equals $\min\{N_\rho,N_{p-3-\rho}\}$ because $N$ is nondecreasing.
\end{proof}

\begin{theorem}\label{thm-last}
Let $h\ge1$. Then $\Cart^{h+1}=0$, and $\rank(\Cart^h)=0$ if $a_0=0$, while
$\rank(\Cart^h)=\binom{M+h}{h+1}$ if $a_0\neq0$. The nilpotency index of $\Cart$ is
exactly $h$ if $a_0=0$ and exactly $h+1$ if $a_0\neq0$.
\end{theorem}

This is \cref{thm-B}.

\begin{proof}
By \cref{thm-higher} for $h\ge2$, and by convention for $h=1$,
$\Img(\Cart^{h-1})=\cU_h^{(h-1)}$, so $\Img(\Cart^h)=\Cart(\cU_h^{(h-1)})$.

If $a_0=0$, every surviving term contains $\alpha_0^{n_0}$ with $n_0\ge1$
(\cref{lem-last-setup}(i)), so $\Cart$ vanishes on $\cU_h^{(h-1)}$ and $\Cart^h=0$.

Assume $a_0\neq0$. By \cref{lem-last-setup}(iii), the basis elements with
$\rho(\omega)<0$ are killed, and $\cU_h^{(h-1)}$ is spanned by them together with
$V_0,\dots,V_{p-3}$. The spaces $W_\rho$ are spanned by pairwise disjoint sets of
monomials, because $\rho=u_1+\dots+u_h+2\sigma$ is determined by the output. Hence
the matrix of $\Cart|_{\cU_h^{(h-1)}}$ is block diagonal with blocks
$\Cart|_{V_\rho}\colon V_\rho\to W_\rho$, and by \cref{prop-block-rank}
\[
  \rank(\Cart^h)=\sum_{\rho=0}^{p-3}\min\{N_\rho,N_{p-3-\rho}\}
  =\sum_{\rho=0}^{M-1}N_\rho+\sum_{\rho=0}^{M-2}N_\rho ,
\]
using $p-3=2M-2$. The first sum counts the tuples
$(u_1,\dots,u_h,\sigma)\in\Z_{\ge0}^{h+1}$ with $u_1+\dots+u_h+2\sigma\le M-1$, the
second the same tuples with right-hand side $M-2$. Sending $\sigma\mapsto\epsilon=2\sigma$
in the first set and $\sigma\mapsto\epsilon=2\sigma+1$ in the second gives a bijection
onto the set of $(u_1,\dots,u_h,\epsilon)\in\Z_{\ge0}^{h+1}$ with
$u_1+\dots+u_h+\epsilon\le M-1$. By \cref{lem-stars},
$\rank(\Cart^h)=\binom{M+h}{h+1}$.

Next, $\Img(\Cart^h)\subseteq\bigoplus_{\rho=0}^{p-3}W_\rho$. Let
$\omega=x^{u_1+pu_2+\dots+p^{h-1}u_h}y^\sigma\,dx$ be a basis element of $W_\rho$. In
the Cartier expansion of $\omega$ the residue of the exponent of $x$ is $u_1$, and
the quantity that must be congruent to $p-1$ is $u_1+n+n_0\le u_1+2\sigma\le\rho\le
p-3$. Hence $\Cart(\omega)=0$, and $\Cart^{h+1}=0$.

Finally, if $a_0=0$ then $\Cart^{h-1}\neq0$ (by \cref{thm-higher} for $h\ge2$, and the identity map is nonzero for $h=1$) while $\Cart^h=0$; if $a_0\neq0$ then $\Cart^h\neq0$ and
$\Cart^{h+1}=0$.
\end{proof}

\begin{corollary}\label{cor-h1}
If $h=1$ and $a_0\ne0$, then $\rank(\Cart)=(p^2-1)/8$ and
$a(\Jac(C_R))=(p-1)(3p-1)/8$.
\end{corollary}

\begin{proof}
For $h=1$, \cref{thm-last} gives $\rank(\Cart)=\binom{M+1}{2}=M(M+1)/2=(p^2-1)/8$, and
$g-(p^2-1)/8=(p-1)(3p-1)/8$.
\end{proof}

\begin{remark}[$p$-rank]\label{rem-prank}
Since $\Cart$ is nilpotent, \cref{prop-a-number} shows that $\Jac(C_R)$ has $p$-rank
$0$. This agrees with the Deuring--Shafarevich formula \cite{Subrao1975}, as
$C_R\to\Pone$ is branched at a single point, and with the supersingularity of
$\Jac(C_R)$ \cite{BHMSSV2016}.
\end{remark}

Combining \cref{prop-h0,thm-higher,thm-last} we obtain the complete sequence.

\begin{corollary}\label{cor-sequence}
Let $h\ge1$. Then
\[
  \rank(\Cart^r)=p^{\,h-r-1}\dbinom{p+r}{r+2}\ \ (1\le r\le h-1),\qquad
  \rank(\Cart^h)=\begin{cases}0,& a_0=0,\\[2pt] \dbinom{M+h}{h+1},& a_0\neq0,\end{cases}
\]
and $\Cart^{h+1}=0$. Thus the sequence of Cartier ranks depends only on $p$, $h$ and
on whether $a_0=0$; the coefficients $a_1,\dots,a_{h-1}$ do not affect it.
\end{corollary}

\begin{table}[ht]
\caption{Cartier-rank sequences $(\rank\Cart^r)_{r\ge1}$ from \cref{cor-sequence},
listed up to the first zero, and the common $a$-number for $h\ge2$.}
\label{tab-sequences}
\centering
\renewcommand{\arraystretch}{1.2}
\begin{tabular}{@{}c r r l l@{}}
\toprule
$(p,h)$ & $g$ & $a(\Jac(C_R))$ & $a_0=0$ & $a_0\neq0$\\
\midrule
$(3,2)$ & $9$    & $5$   & $(4,0)$             & $(4,1,0)$ \\
$(3,3)$ & $27$   & $15$  & $(12,5,0)$          & $(12,5,1,0)$ \\
$(3,4)$ & $81$   & $45$  & $(36,15,6,0)$       & $(36,15,6,1,0)$ \\
$(3,5)$ & $243$  & $135$ & $(108,45,18,7,0)$   & $(108,45,18,7,1,0)$ \\
$(5,2)$ & $50$   & $30$  & $(20,0)$            & $(20,4,0)$ \\
$(5,3)$ & $250$  & $150$ & $(100,35,0)$        & $(100,35,5,0)$ \\
$(7,2)$ & $147$  & $91$  & $(56,0)$            & $(56,10,0)$ \\
$(7,3)$ & $1029$ & $637$ & $(392,126,0)$       & $(392,126,15,0)$ \\
$(11,2)$ & $605$ & $385$ & $(220,0)$           & $(220,35,0)$ \\
$(13,2)$ & $1014$& $650$ & $(364,0)$           & $(364,56,0)$ \\
\bottomrule
\end{tabular}
\end{table}

\begin{remark}[Computational check]\label{rem-computation}
Independently of the proofs, we computed the Cartier--Manin matrix directly from
\eqref{eq-yj} and \eqref{eq-C2}, without using \cref{prop-expansion}, together with
the ranks of its semilinear powers (\cref{rem-semilinear}), for $236$ curves: $h=0$
with $p\le11$; $h=1$ with $p\le13$; and every $(p,h)$ in \cref{tab-sequences}. The
coefficients were random in $\F_p$, and also in $\F_{p^2}$ when $g\le400$, with both
$a_0=0$ and $a_0\neq0$ and with $a_1,\dots,a_{h-1}$ either random or zero. In every
case the ranks agree with \cref{prop-h0,cor-sequence}.
\end{remark}

\section{Ekedahl--Oort types, extremality and characteristic two}\label{sec-consequences}

\begin{remark}[Ekedahl--Oort types]\label{rem-EO}
Let $h\ge1$. By \cref{thm-last}, $\rank(\Cart^h)$ vanishes on the locus $a_0=0$ and
not on its complement. Since the rank sequence is an invariant of the
Ekedahl--Oort type, no Ekedahl--Oort type occurs on both loci. In particular the
condition $a_0=0$ is intrinsic: if $C_R\cong C_{R'}$ over $k$, then $a_0=0$ if and
only if $a'_0=0$. For $h\ge2$ the two loci have the same $a$-number
(\cref{thm-hge2}), so $\rank(\Cart^h)$ is an invariant beyond the $a$-number that
detects the divisor $\{a_0=0\}$ in the space of coefficients.
\end{remark}

Farnell and Pries \cite{FarnellPries2013} obtained families with constant
Cartier--Manin rank under the condition that $p\equiv1$ modulo the pole orders. For
the family \eqref{eq-family} with $h\ge1$ the pole order is $p^h+1>p$, so that
condition fails; here the constancy of the ranks of $\Cart^r$ for $r\le h-1$ comes
from the image spaces $\cU_h^{(r)}$ of \cref{sec-higher}.

\begin{corollary}[Extremal Cartier ranks]\label{cor-extremal-ranks}
Let $h\ge2$, $d=p^h+1$ and $g=p^h(p-1)/2$. For a polynomial $f\in k[x]$ of degree
$d$ let $\Cart_f$ be the Cartier operator of $y^p-y=f(x)$. Then
\[
  \rank(\Cart_f^r)\ge p^{h-r-1}\binom{p+r}{r+2}\qquad(1\le r\le h-1),
\]
and every member $C_R$ attains these bounds simultaneously. In particular
\[
  \max_{\deg f=d}a(\Jac(C_f))=\frac{p^{h-1}(p-1)(2p-1)}6,
\]
and every $C_R$ attains this maximum. If $a_0=0$, then $C_R$ maximizes
$\dim_k\ker(\Cart_f^r)$ for every $r\ge1$ among all polynomials $f$ of degree $d$:
\[
  \max_{\deg f=d}\dim_k\ker(\Cart_f^r)=
  \begin{cases}
    g-p^{h-r-1}\dbinom{p+r}{r+2},&1\le r\le h-1,\\[3pt]
    g,&r\ge h.
  \end{cases}
\]
\end{corollary}

\begin{proof}
By Groen's explicit upper bound \cite[Theorem~6.16]{Groen2024}, specialized to a
one-point cover of $\Pone$ with ramification break $d$, and the sharpness at the
monomial member explained in \cref{rem-groen},
$\dim_k\ker(\Cart_f^r)\le\dim_k\ker(\Cart_{x^d}^r)$ for all $r\ge1$. All these curves have genus $g$. Applying \cref{cor-sequence} to the
monomial member $y^p-y=x^d$ and to $C_R$ gives the inequalities and the equality
for $1\le r\le h-1$. If $a_0=0$, then $\Cart^h=0$ on $H^0(C_R,\Omg)$, so the kernel
of $\Cart^r$ is everything for $r\ge h$.
\end{proof}

\begin{remark}[Comparison with the generic $a$-number]\label{rem-generic-a}
Cais and Ulmer \cite[Theorem~1.3]{CaisUlmer2026} determine the $a$-number of
$y^p-y=f(x)$ for $f$ in a nonempty Zariski open subset of the space of polynomials
of degree $d$, $p\nmid d$; it equals the Booher--Cais lower bound. For
$d=p^h+1$ with $h\ge2$, put $P=p^{h-2}$ and $K_\ell=(p^2+1-2\ell p)/2$ for
$1\le\ell\le M$. Then $d=p^2P+1$ and $0<K_\ell<p^2$, so their formula becomes
\begin{align*}
  a_{\mathrm{gen}}
  &=\sum_{\ell=1}^{M}
    \left(\left\lfloor\frac{(p-\ell)d}{p}\right\rfloor
    -\left\lfloor\frac{K_\ell d}{p^2}\right\rfloor\right)
  =P\sum_{\ell=1}^{M}\bigl(p(p-\ell)-K_\ell\bigr)
  =\frac{p^{h-2}(p-1)^2(p+1)}4.
\end{align*}
By \cref{thm-hge2}, every member of the family satisfies
\[
  a(\Jac(C_R))-a_{\mathrm{gen}}=\frac{p^{h-2}(p-1)(p^2-2p+3)}{12}>0.
\]
For $(p,h)=(3,2)$, $(3,3)$, $(5,2)$ and $(7,2)$ we also checked numerically that
the minimum of $a(\Jac(C_f))$ over random polynomials $f$ of degree $d$ with
coefficients in $\F_{p^3}$ or $\F_{p^2}$ equals $a_{\mathrm{gen}}=4,12,24,72$.
\end{remark}

\begin{remark}[Outside the additive family]\label{rem-special}
Let $p=5$, $h=2$ and $f_0=x^2+2x^6+x^{26}$, which is $xR(x)$ for
$R=x+2x^5+x^{25}$. Computing the Cartier--Manin matrix as in
\cref{rem-computation}, with coefficients in $\F_5$, gives
\[
\begin{array}{c|cc}
 f & a(\Jac(C_f)) & (\rank\Cart_f^r)_{r\ge1}\\ \hline
 f_0 & 30 & (20,4,0)\\
 f_0+x^{13} & 26 & (24,10,2,0)\\
 f_0+x^3 & 30 & (20,6,1,0)
\end{array}
\]
whereas the generic value is $a_{\mathrm{gen}}=24$ (\cref{rem-generic-a}). Thus the
$a$-number varies among curves with the same pole order, and the maximal $a$-number
does not characterize the family: the curve $y^5-y=f_0+x^3$ attains it, but its
rank sequence is neither $(20,0)$ nor $(20,4,0)$, so by \cref{cor-sequence} it is not
isomorphic over $k$ to any $C_R$. By \cref{rem-normalization}, adding an additive
polynomial to $f_0$ does not change the isomorphism class; for instance
$f_0+x^5$ and $f_0+2x+3x^{25}$ again give $(20,4,0)$.

The rank sequence itself is not specific to the family either. For example,
$y^3-y=x^{28}+x^{19}$ has rank sequence $(12,5,1,0)$, which is the sequence of the
members with $(p,h)=(3,3)$ and $a_0\neq0$, and $y^3-y=x^{82}+x^{13}$ has
$(36,15,6,0)$, which is that of the members with $(p,h)=(3,4)$ and $a_0=0$. We do
not know whether these two curves are isomorphic over $k$ to members of the family.
\end{remark}

\begin{remark}[Characteristic two]\label{rem-char2}
The hypothesis $p>2$ is essential for the dichotomy of \cref{thm-B}. If $p=2$ and
$h\ge1$, then $a_0x^2\equiv a_0^{1/2}x$ modulo $w^2-w$ with $w=a_0^{1/2}x$, and
the resulting linear term can be removed by a translation $x=X+\lambda$ as in
\cref{rem-normalization}. Hence in characteristic $2$ every $C_R$ is isomorphic to a
member with $a_0=0$. Moreover $C_R$ is then a double cover of $\Pone$ branched at a
single point with ramification invariant $2^h+1$, and by Elkin and Pries
\cite{ElkinPries2013} the $2$-torsion group scheme of the Jacobian of a
hyperelliptic curve in characteristic $2$ depends only on the ramification invariants
at the branch points. So for $p=2$ and fixed $h\ge2$ the Ekedahl--Oort type is
constant on the whole family.
\end{remark}

\begin{question}\label{q-EO}
Let $p$ be odd and $h\ge1$. Determine the Ekedahl--Oort type of $\Jac(C_R)[p]$.
After normalizing $a_h=1$, is it constant on each of the two loci $a_0=0$ and
$a_0\ne0$ as the remaining coefficients vary? Which polynomials $f$ of degree
$p^h+1$ give the same Cartier-rank sequence, or the same Ekedahl--Oort type, as the
family?
\end{question}

For $h=1$ and $a_0=0$ the curve is Hermitian, and its Ekedahl--Oort type is known
\cite{PriesWeir2015}. Any variation within either locus is invisible to the ranks of
the powers of $\Cart$.

\section{Curves over finite fields: proof of \texorpdfstring{\cref{thm-C}}{Theorem C}}\label{sec-finite}

The link between the Cartier operator and maximality is the following criterion of
Garcia and Tafazolian.

\begin{theorem}[{\cite[Theorem~3.3]{GarciaTafazolian2008}; see also \cite{MontanucciSpeziali2018}}]\label{thm-GT}
Let $q=p^n$ and let $X$ be a curve over $\F_{q^2}$. If $X$ is $\F_{q^2}$-maximal or
$\F_{q^2}$-minimal, then $\Cart^n=0$ on $H^0(X\otimes\Fbar,\Omg)$.
\end{theorem}

\begin{proof}[Proof of \cref{thm-C}]
Let $X$ be as in \cref{thm-C} and assume that it is $\F_{p^{2n}}$-maximal or
$\F_{p^{2n}}$-minimal. By \cref{thm-GT}, $\Cart^n=0$ on
$H^0(X\otimes\Fbar,\Omg)\cong H^0(C_R,\Omg)$. By \cref{thm-last}, applied over
$k=\Fbar$, the nilpotency index of $\Cart$ on $H^0(C_R,\Omg)$ is $h$ if $a_0=0$ and
$h+1$ if $a_0\neq0$; by \cref{rem-EO} this does not depend on the choice of $R$ with
$X\otimes\Fbar\cong C_R$. Hence $n\ge h$, and $n\ge h+1$ if $a_0\ne0$.
\end{proof}

\begin{remark}[Comparison with known criteria]\label{rem-ITT}
Let $R\in\F_q[x]$ and let $\F_q\subseteq\F_{p^{2f}}$. Ito, Tatematsu and Tsushima
proved that if the curve $y^p-y=xR(x)$ is $\F_{p^{2f}}$-maximal, then $h\le f$
\cite[Proposition~4.13]{ItoTatematsuTsushima2026}; their proof uses the explicit
$L$-polynomial of \cite{TakeuchiTsushima2024}. \Cref{thm-C}, applied with $n=f$,
recovers this inequality; it applies equally to minimal curves and to every
$\F_{p^{2n}}$-form of $C_R$, and it excludes $n=h$ when $a_0\ne0$. For maximal curves of the form \eqref{eq-family}
over $\F_{p^{2h}}$ more is known: such a curve is $\F_{p^{2h}}$-maximal if and only
if $a_0=\dots=a_{h-1}=0$ and $a_h^{p^h}+a_h=0$
\cite{OzbudakSaygi2016}, \cite[Proposition~4.14]{ItoTatematsuTsushima2026}. In
particular the bound $n\ge h$ of \cref{thm-C} is attained for every $p$ and $h$.
\end{remark}

\begin{example}[Sharpness of $n\ge h+1$]\label{ex-sharp}
For a curve $X\colon y^p-y=f(x)$ with $f\in\F_\ell[x]$ one has
$\#X(\F_\ell)=1+p\cdot\#\{x\in\F_\ell:\ \operatorname{Tr}_{\F_\ell/\F_p}(f(x))=0\}$.
Direct enumeration gives the following members with $a_0\neq0$ that are maximal or
minimal over $\F_{p^{2(h+1)}}$.
\begin{enumerate}[label=\textup{(\alph*)}, leftmargin=2em]
  \item $(p,h)=(3,1)$. Let $\zeta\in\F_{81}$ be a root of $T^4+2T^3+2\in\F_3[T]$.
    The curve $y^3-y=\zeta x^2+\zeta^7x^4$ has $136=81+1+2\cdot3\cdot9$ points over
    $\F_{81}$, so it is $\F_{81}$-maximal.
  \item $(p,h)=(3,2)$. The curve $y^3-y=x^2-x^{10}$ has $1216=729+1+2\cdot9\cdot27$
    points over $\F_{3^6}$, so it is $\F_{3^6}$-maximal.
  \item $(p,h)=(5,2)$. The curve $y^5-y=x^2+2x^{26}$ has
    $3126=5^6+1-2\cdot50\cdot125$ points over $\F_{5^6}$, so it is
    $\F_{5^6}$-minimal.
\end{enumerate}
By \cref{thm-C}, none of these curves is maximal or minimal over a smaller field
$\F_{p^{2n}}$ with $n\le h$. As a consistency check, we found by exhaustive
enumeration that none of the $64$ curves $y^3-y=a_0x^2+a_1x^4$ with
$a_0,a_1\in\F_9^\times$ is $\F_9$-maximal or $\F_9$-minimal, while among the $6400$
such curves with $a_0,a_1\in\F_{81}^\times$ exactly $160$ are $\F_{81}$-maximal and
$80$ are $\F_{81}$-minimal.
\end{example}

\begin{question}\label{q-sharp}
For which $(p,h)$ is the bound $n\ge h+1$ of \cref{thm-C} attained when $a_0\ne0$?
For $(p,h)=(3,3)$, none of the members with coefficients in $\F_3$ and $a_0\neq0$ is
maximal or minimal over $\F_{3^8}$.
\end{question}

\end{document}